\documentclass[12pt,a4paper]{amsart}
\usepackage{amsfonts}
\usepackage{amsthm}
\usepackage{amsmath}
\usepackage{amssymb}
\usepackage{amscd}
\usepackage[latin2]{inputenc}
\usepackage{t1enc}
\usepackage[mathscr]{eucal}
\usepackage{indentfirst}
\usepackage{graphicx}
\usepackage{graphics}
\usepackage{pict2e}
\usepackage{epic}
\numberwithin{equation}{section}
\usepackage[margin=2.9cm]{geometry}
\usepackage{epstopdf} 

\usepackage{soul}
\usepackage[dvipsnames]{xcolor}

\theoremstyle{plain}
\newtheorem{Th}{Theorem}[section]
\newtheorem{Lemma}[Th]{Lemma}
\newtheorem{Cor}[Th]{Corollary}
\newtheorem{Prop}[Th]{Proposition}

\theoremstyle{definition}
\newtheorem{Def}[Th]{Definition}
\newtheorem{Conj}[Th]{Conjecture}
\newtheorem{Rem}[Th]{Remark}
\newtheorem{?}[Th]{Problem}
\newtheorem{Ex}[Th]{Example}

\begin{document}

\title[The rotor-routing torsor and the Bernardi torsor]{The rotor-routing torsor and the Bernardi torsor disagree for every non-planar ribbon graph}

\author[Changxin Ding]{Changxin Ding}

\address{Brandeis University \\ Department of Mathematics \\
Waltham MA 02453} 

\email{dcx@brandeis.edu}



\begin{abstract}
Let $G$ be a ribbon graph. Matthew Baker and Yao Wang proved that the rotor-routing torsor and the Bernardi torsor for $G$, which are two torsor structures on the set of spanning trees for the Picard group of $G$, coincide when $G$ is planar. We prove the conjecture raised by them that the two torsors disagree when $G$ is non-planar.  

\end{abstract}

\maketitle

\section{Introduction} 
This paper is aimed at completing the proof of the following conjecture proposed by Matthew Baker and Yao Wang in \cite{BW}. 

\begin{Conj}\label{Conj} Let $G$ be a connected ribbon graph without loops or multiple edges. The Bernardi and rotor-routing torsors $b_v$\footnote{In \cite{BW}, the Bernardi torsor is denoted by $\beta_v$. However, the closely related Bernardi bijection from spanning trees to break divisors is denoted by $\beta$ or $\beta_{(v,e)}$. So we change the notation here to avoid ambiguity. } and $r_v$ agree for all vertices $v$ if and only if $G$ is planar. 
\end{Conj}

The "if" part of the conjecture has been proved in \cite{BW} and we will show the "only if" part is also true:

\begin{Th} \label{main} Let $G$ be a connected non-planar ribbon graph without loops or multiple edges. The Bernardi and rotor-routing torsors $b_v$ and $r_v$ do not agree for some vertex $v$ of $G$. 
\end{Th}

What does the conjecture mean? Let us give a brief introduction here. One can also see the introduction in \cite{BW}. 

Let $G$ be a connected graph on $n$ vertices. The \emph{Picard group} $\text{Pic}^{0}(G)$ of $G$ (also called the sandpile group, Jacobian group, or critical group) is a discrete analogue of the Jacobian of a Riemann surface. The cardinality of $\text{Pic}^{0}(G)$ is the determinant of any $(n-1)\times(n-1)$ principal sub-minor of the Laplacian matrix of $G$ and hence equals the cardinality of the set $S(G)$ of spanning trees of $G$ by Kirchhoff's Matrix-Tree Theorem. 

It is natural to look for bijections between the group $\text{Pic}^{0}(G)$ and the set $S(G)$. However, since one of the objects is a group, people can ask for more than a bijection. Indeed, one can define a \emph{torsor} for a group $P$ to be a set $S$ together with a simply transitive action of $P$ on $S$. If the set $S(G)$ has a torsor structure for the group $\text{Pic}^{0}(G)$, then their cardinalities are equal. 

We are interested in two kinds of torsors. Both of them are defined for a \emph{ribbon graph} together with a basepoint vertex $v$. One can think of the ribbon graph as a graph drawn on a closed orientable surface and hence having a cyclic ordering of the edges around each vertex. 

Holroyd \textit{et al.} \cite{H} defined the \emph{rotor-routing torsor} $r_v$. Then Melody Chan, Thomas Church, and Joshua A. Grochow \cite{CCG} proved that the rotor-routing torsor is independent of the basepoint if and only if $G$ is a planar ribbon graph.

Matthew Baker and Yao Wang \cite{BW} observed that one could use the Bernardi bijection
in \cite{Bernardi} to define the \emph{Bernardi torsor} $b_v$ and proved that the Bernardi torsor is independent of the basepoint if and only if $G$ is a planar ribbon graph. Moreover, they proved that in the planar case $b_v$ and $r_v$ agree. Then they raised Conjecture~\ref{Conj}, which is the target of this paper.

This paper can be viewed as a complement to \cite{BW}. We adopt notation, terminology, and some useful lemmas from \cite{BW}, which are reviewed briefly in Section~\ref{Background}. In Section~\ref{main lemma}, we prove some technical lemmas, which categorize all the non-planar ribbon graphs into two types, called \emph{type A} and \emph{type B}. In either case, the graph is decomposed into two parts so that in Section~\ref{Proof} we can handle the computation of the rotor-routing and Bernardi process and hence prove Theorem~\ref{main}. 

In an independent work, Farbod Shokrieh and Cameron Wright also solved this conjecture \cite{ShW}.

\section{Background}\label{Background}

In this section we introduce notation and review briefly the rotor-routing torsor, the Bernardi torsor, and some useful lemmas from \cite{BW} . We refer to \cite{BW} for more details. 

For any positive integer $n$, we denote by $[n]$ the set $\{1,2,\cdots,n\}$.

Let $G$ be a graph, by which we mean a finite connected graph, possibly with loops and multiple edges. We use $V(G)$ to denote the vertex set of $G$, and $E(G)$ the edge set of $G$. Recall that a \emph{divisor} on $G$ is a formal sum of vertices with integer coefficients, written as $\sum_{v\in V(G)}a_v (v)$, where $a_v\in\mathbb{Z}$. The \emph{degree} of this divisor is $\sum_{v\in V(G)}a_v$, and the set of divisors of degree $d$ is denoted by $\text{Div}^d(G).$ Given an oriented edge $\overrightarrow{e}=\overrightarrow {uv}$, we denote by $\partial \overrightarrow{e}$ the divisor $(v)-(u)$. The group of \emph{principal divisors} on $G$, denoted by $\text{Prin}(G)$, is the subgroup of $\text{Div}^0(G)$ generated by $\{\sum_{\overrightarrow{uv} \text{ is incident to } v}\partial(\overrightarrow{uv}):v\in V(G)\}$. We say that two divisors $D$ and $D'$ are \emph{linearly equivalent}, written $D\sim D'$, if $D-D'\in\text{Prin}(G)$.  We denote the linear equivalence class of a divisor $D$ by $[D]$. The \emph{Picard group} of $G$ is defined to be $\text{Pic}^{0}(G)=\text{Div}^0(G)/\text{Prin}(G)$. 

To define the two torsors, the graph $G$ must be endowed with a \emph{ribbon structure}, meaning the edges around each vertex of $G$ have a cyclic ordering. We call a graph with such a structure a \emph{ribbon graph}. In this paper, we draw a ribbon graph on a plane (possibly with edges crossing) such that around each vertex the counterclockwise orientation indicates the cyclic ordering of the edges. For example, in Figure~\ref{fig1}, the cyclic ordering around the vertex $c$ is $(ca,cd,cf)$ and the cyclic ordering around the vertex $b$ is $(bf, ba, bd)$. A \emph{non-planar ribbon graph} is a ribbon graph that cannot be drawn on a plane with no crossings respecting the ribbon structure. In the example, by ignoring the ribbon structure one can draw the graph on a plane with no crossings, but as a ribbon graph it is non-planar. 

For the \emph{rotor-routing torsor}, we need to recall the meaning of the notation $T'=((x)-(y))_y(T)$, where $T,T'$ are spanning trees and $x,y$ are vertices.  The vertex $y$ is viewed as a fixed sink. The input is the spanning tree $T$. The output $T'=((x)-(y))_y(T)$ is determined by an algorithm called \emph{rotor-routing process}. Initially we orient the edges of $T$ towards $y$ and put a chip at $x$. Then in each step of the rotor-routing process, we (i) turn the rotor around the chip, meaning rotating the unique oriented edge whose tail locates the chip to the next one according to the ribbon structure, and (ii) move the chip along the newly oriented edge to its head. In each step, the set of the oriented edges is called a \emph{rotor configuration}. It is proved in \cite{H} that at the end of this process the chip reaches the sink $y$ and the rotor configuration forms a spanning tree, which is the output $T'$. This defines how an element $(x)-(y)$ of $\text{Div}^0(G)$ acts on the set $S(G)$ of the spanning trees of $G$. Because $\text{Div}^0(G)$ is freely generated by the divisors $\{(x)-(y):x\in V(G)\backslash\{y\}\}$, one can define the action of $\text{Div}^0(G)$ on $S(G)$ in a natural way. It is proved in \cite{H} that this action descends to a simply transitive action $r_y$ of $\text{Pic}^{0}(G)$ on $S(G)$, which is called the rotor-routing torsor.

In Figure~\ref{fig1}, let $T$ be the spanning tree $\{ca,cf,ab,bd\}$. If one puts the sink at $d$ and the chip at $c$, then in one step the chip will reach the sink $d$ and  hence get the spanning tree $T'=((c)-(d))_d(T)=\{cd,cf,ab,bd\}$. See also Figure~\ref{fig9-1} for a more complicated example. 

For the \emph{Bernardi torsor}, we need to recall the meaning of the notation $\beta_{(v,e)}(T)$, where $T$ is a spanning tree, $v$ is a vertex, $e$ is an edge incident to $v$, and $\beta_{(v,e)}$ is the \emph{Bernardi bijection} from $S(G)$ to the set of \emph{break divisors} $B(G)$. The break divisors are certain divisors of degree $g$ and have the property (cf. \cite[Theorem 4.21]{AKBS}) that each linear equivalence class of $\text{Div}^g(G)$ contains exactly one break divisor, where $g=\#E(G)-\#V(G)+1$ is the nullity\footnote{In \cite{BW}, $g$ is called the combinatorial genus of $G$.} of $G$. The map $\beta_{(v,e)}$ is induced by the \emph{Bernardi process}. Informally the Bernardi process uses a tour on the surface where the ribbon graph $G$ is embedded. The tour begins with $(v,e)$, goes along the edges in the spanning tree $T$, and cuts through the edges not in $T$. Note that in the process each edge not in $T$ is cut twice. Each time we first cut an edge we put a chip at the corresponding endpoint. When the process is over, we put totally $g$ chips and hence get a divisor of degree $g$, denoted by $\beta_{(v,e)}(T)$. It is proved implicitly in \cite{Bernardi} that $\beta_{(v,e)}$ is a bijection. The paper \cite{BW} gives another proof, uses $\beta_{(v,e)}$ to define the Bernardi torsor $b_v$, and proves that $b_v$ does not depend on the choice of $e$ for any fixed vertex $v$. In brief, the action $b_v$ of $\text{Pic}^0(G)$ on $S(G)$ is defined by $[D]\cdot T=T'\Leftrightarrow  \beta_{(v,e)}(T)+D\sim\beta_{(v,e)}(T')$, where $D\in \text{Div}^0(G)$.

In Figure~\ref{fig1}, the red tour shows how the Bernardi process goes and we get $\beta_{(d,dc)}(T)=(d)+(b)$ and $\beta_{(d,dc)}(T')=(f)+(c)$. So for $D=(c)-(d)+(f)-(b)$ one gets $[D]\cdot T=T'$. See also Figure~\ref{fig9-2} for a more complicated example. 


\begin{figure}[h]
            \centering
            \includegraphics[scale=0.6]{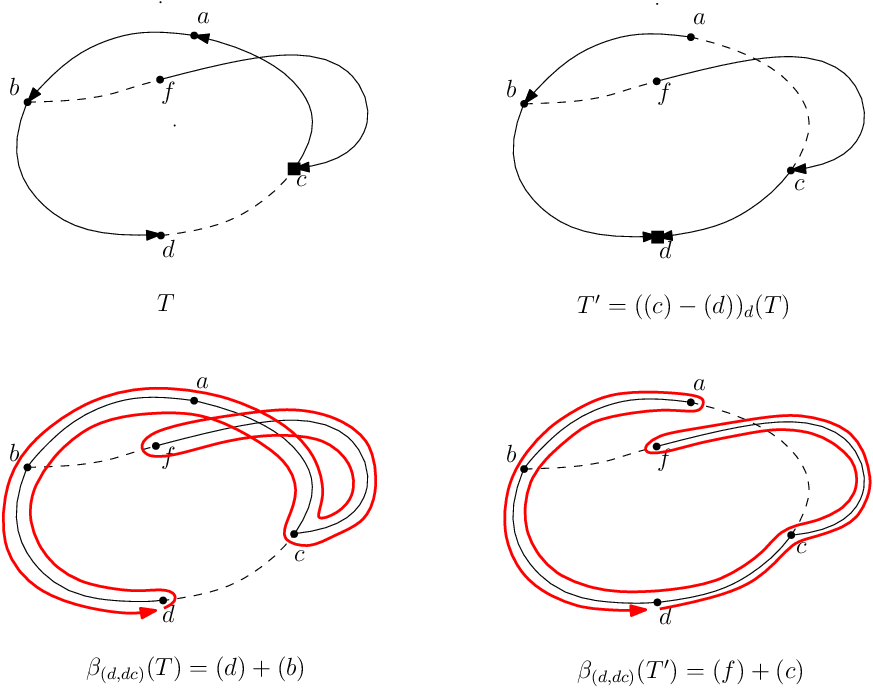}
            \caption{An example of the rotor-routing and Bernardi process. The graph $G$ has five vertices $\{a,b,c,d,f\}$ and six edges. An edge is dashed if and only if it is not in the spanning tree. In the top two pictures, the oriented edges indicate the rotor configuration and the square indicates the chip. In the bottom two pictures, the red tour indicates the Bernardi process. }
            \label{fig1}
\end{figure}

Comparing the two torsors, we get the following lemma, which is obviously true and used a few times in \cite{BW}, although \cite{BW} does not state it as lemma. 

\begin{Lemma} \label{l1}
Fix a vertex $y$ of the graph $G$ and an edge $e$ incident to $y$. The Bernardi and rotor-routing torsors $b_y$ and $r_y$ agree if and only if  $\beta_{(y,e)}(T')-\beta_{(y,e)}(T)\sim (x)-(y)$ for any vertex x and any spanning tree $T$, where $T'=((x)-(y))_y(T)$.
\end{Lemma}

The next formula is in the proof of Theorem 4.1 in \cite{BW}. 

\begin{Lemma}\label{l2}
In the graph $G$, denote the cyclic ordering of the edges around a vertex $v$ by $(e_1,a_1,\ldots,a_k,e_2,b_1,\ldots,b_l)$. Then $\beta_{(v,e_2)}(T)-\beta_{(v,e_1)}(T)\sim \partial\overrightarrow{e_1}+\partial\overrightarrow{a_1}+\cdots+\partial\overrightarrow{a_k}$, where $T$ is a spanning tree of $G$ and all the oriented edges are oriented away from $v$. 
\end{Lemma}

The next lemma is Lemma 5.3 in \cite{BW}, but we state it in a different way. The original proof still works for our statement. 

\begin{Lemma}\label{l3}
Let $B$ be a partial orientation of $G$, meaning each edge in the graph is oriented in either way or not oriented. 
If $\sum_{\overrightarrow{e}\in B}\partial \overrightarrow{e}\sim 0$ and $B$ contains no directed cycle, then $B$ is a disjoint union of directed cuts in $G$. In particular, if we further assume that $B$ contains at least one oriented edge, then $B$ contains a directed cut.  

\end{Lemma}

To conclude this section, let us prove that the two torsors $b_d$ and $r_d$ disagree in the above example(Figure~\ref{fig1}), which serves as a prototype of Proposition~\ref{A}. In order to show the two torsors $b_d$ and $r_d$ disagree it suffices to prove that $\beta_{(d,dc)}(T')-\beta_{(d,dc)}(T)\sim (c)-(d)$ is false by Lemma~\ref{l1}. This is equivalent to $(f)-(b) \nsim 0$. We take $B=\{\overrightarrow{bf}\}$ in Lemma~\ref{l3} and hence $\sum_{\overrightarrow{e}\in B}\partial \overrightarrow{e}=(f)-(b)$. Assume $(f)-(b) \sim 0$, then the edge $bf$ should be a cut by the lemma, which leads to a contradiction.

\section{Technical Lemmas:a decomposition of non-planar ribbon graph}\label{main lemma}
We will use Lemma~\ref{l1} to prove the main theorem. In general it is very hard to compute 
$T'=((x)-(y))_y(T)$ and $\beta_{(y,e)}(T')-\beta_{(y,e)}(T)$. Our strategy is to decompose a graph into two parts so that the computation is easier. 

\begin{Def}\label{wedge}
Let $G=(V,E)$ be a graph and $G_1=(V_1,E_1),G_2=(V_2,E_2)$ be two subgraphs. If $E$ is the
disjoint union of $E_1$ and $E_2$ and $V_1 \cap V_2=\{ c \}$, then we call $G$ the \emph{wedge sum} of $G_1$ and $G_2$ at $c$, denoted by $G_1 \vee_c G_2$. 
\end{Def}

Note that if $G$ is connected, then $G_1$ and $G_2$ are also connected.

To decompose a non-planar ribbon graph in the sense of Definition~\ref{wedge}, we start by introducing a classical result. See, e.g., \cite{Y}(Lemma 30) for a proof.

\begin{Lemma}\label{type12}
For any non-planar ribbon graph $G$, there exists a subgraph (with the inherited ribbon structure) which is of either type I or type II (defined as follows).
\end{Lemma}

\begin{Def}(See Figure~\ref{fig2})
(1) We say that a ribbon graph is of type I if it consists of three paths whose vertex sequences are $(c, a_1, \cdots, a_n, b)$, $(c, d_1, \cdots, d_m, b)$, and $(c, f_1, \cdots, f_k, b)$, respectively, where all the vertices are distinct and $n,m,k$ could be $0$, and the cyclic ordering of the edges around each vertex is indicated as in the figure. To be precise, the cyclic ordering around $c$ is $(ca_1,cd_1,cf_1)$ and the cyclic ordering around $b$ is $(bf_k,ba_n,bd_m)$. 
\end{Def}

(2) We say that a ribbon graph is of type II if it consists of two cycles  whose vertex sequences are $(c, a_1, \cdots, a_n, c)$ and $(c, f_1, \cdots, f_k, c)$, respectively, where all the vertices are distinct and $n, k$ could be $0$, and the cyclic ordering of the edges around each vertex is indicated as in the figure. To be precise, the cyclic ordering around $c$ is $(ca_1,cf_k,ca_n,cf_1)$.

\begin{figure}[h]
            \centering
            \includegraphics[scale=0.6]{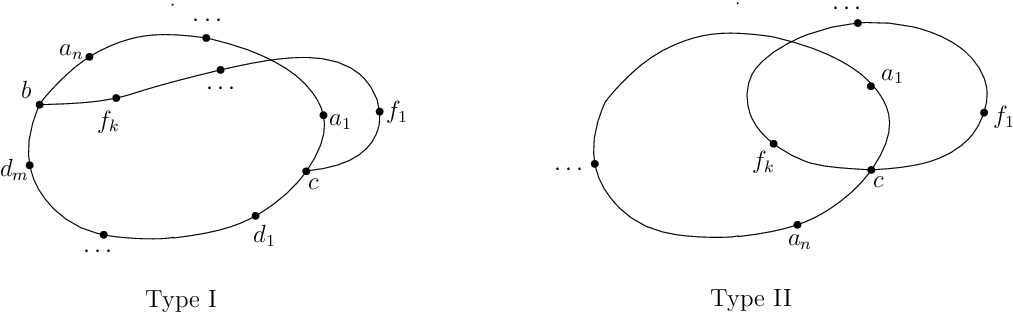}
            \caption{Type I and II}
            \label{fig2}
\end{figure}

Furthermore, we want to show any non-planar ribbon graph $G$ is of either type A or type B as defined below. 

\begin{Def}\label{type1a}(See Figure~\ref{fig3})
\item (1) Assume a ribbon graph $G$ contains a subgraph $H$ of type I. Denote the cyclic ordering of the edges around $c$ in $G$ by $(ca_1,cx_1,\cdots,cx_N,cd_1,cy_1,\cdots,cy_M)$, 
where $N\ge 0$, $M>0$, and $f_1\in\{y_1,\cdots,y_M\}$. Let $G_1$ be the subgraph of $G$ induced by all the edges that are connected to one of the edge $cx_i$'s by a path where $c$ can only be used at the two endpoints. Let $G_2$ be the subgraph of $G$ induced by the edges not in $G_1$. In the case that $G_i$($i\in\{1,2\}$) does not contain any edge, set $G_i$ to be the one single vertex graph $c$. We call $(H,G_1,G_2)$ the \emph{$H$-decomposition}\footnote{Strictly speaking, in Definition~\ref{type1a} and Definition~\ref{type2b} the $H$-decomposition depends not only on $G$ and $H$ but also on $c$ and $ca_1$, so there are more than one $H$-decomposition when $G$ and $H$ are given. However, any $H$-decomposition will work for the remaining part of the paper. } of $G$.

\item (2) We call a non-planar ribbon graph $G$ of \emph{type A} if it contains a subgraph $H$ of type I such that the subgraph $G_1$ in the $H$-decomposition does not contain any vertex in $V(H)\backslash \{ c\}$.     
\end{Def}
\begin{figure}[h]
            \centering
            \includegraphics[scale=0.6]{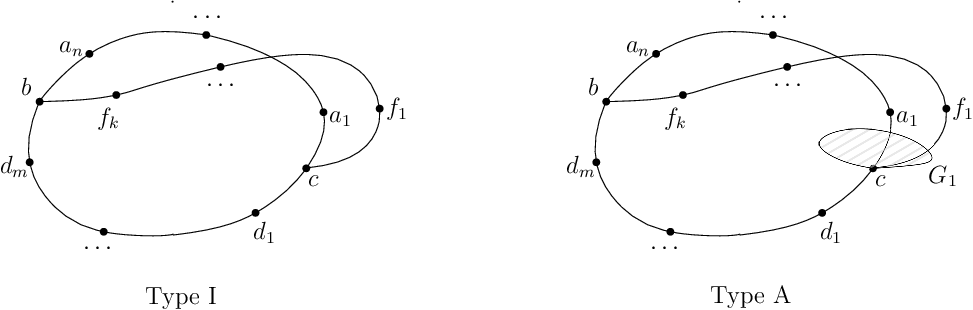}
            \caption{Type I and A}
            \label{fig3}
\end{figure}

\begin{Rem}
Adopt the notation $G,H,G_1,G_2$ of Definition~\ref{type1a}.
\item (1) All the edge $cx_i$'s are in $G_1$. When $N=0$, $G_1$ has no edge. 
\item (2) If $G$ is of type A, then $G_2$ contains $H$. 
\item (3) $G=G_1 \vee_c G_2$. 
\end{Rem}

\begin{Def}\label{type2b} (See Figure~\ref{fig4})
\item (1) Assume a ribbon graph $G$ contains a subgraph $H$ of type II. Denote the cyclic ordering of the edges around $c$ in $G$ by $(ca_1,cx_1,\cdots,cx_N,ca_n,cy_1,\cdots,cy_M)$, 
where $N>0$, $M>0$, and $f_k\in\{x_1,\cdots,x_N\}, f_1\in\{y_1,\cdots,y_M\}$. Let $G_1$ be the subgraph of $G$ induced by all the edges that are connected to one of the edges $cx_i$'s by a path where $c$ can only be used at the two endpoints. Let $G_2$ be the subgraph of $G$ induced by the edges not in $G_1$.  In the case that $G_2$ does not contain any edge, set $G_2$ to be the one single vertex graph $c$. We call $(H,G_1,G_2)$ the \emph{$H$-decomposition} of $G$.

\item (2) We call a non-planar ribbon graph $G$ of \emph{type B} if it contains a subgraph $H$ of type II such that the subgraph $G_1$ in the $H$-decomposition does not contain any of the vertices $a_1,\cdots, a_n$.     
\end{Def}

\begin{figure}[h]
            \centering
            \includegraphics[scale=0.6]{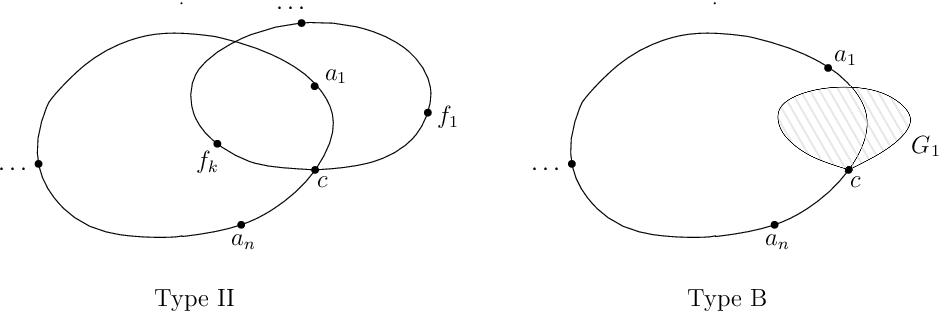}
            \caption{Type II and B}
            \label{fig4}
\end{figure}

\begin{Rem} Adopt the notation $G,H,G_1,G_2$ of Definition~\ref{type2b}.
\item (1) All the edge $cx_i$'s are in $G_1$, and the cycle $\{cf_1, f_1f_2, \cdots, f_kc\}$ is also in $G_1$. 
\item (2) If $G$ is of type B, then $G_2$ contains the cycle $\{ca_1, a_1a_2, \cdots, a_nc\}$. 
\item (3) $G=G_1 \vee_c G_2$. 
\end{Rem}

\begin{Lemma}\label{1toa}
If a non-planar ribbon graph $G$ contains a subgraph $H$ of type I, then $G$ is of type A.
\end{Lemma}
\begin{proof}
Let $G_1$ be as in the first part of Definition~\ref{type1a}. If $G_1$ does not contain any vertex in $V(H)\backslash \{ c\}$, then $G$ is already of type A. Otherwise, by the construction of $G_1$ there exists a path with vertex sequence $(z_0,z_1,z_2,\cdots,z_l)$ where $z_0=c,z_1\in\{x_1,\cdots,x_N\}$, $z_1,z_2,z_3,\cdots,z_{l-1}\neq c$, and $z_l\in V(H)\backslash \{ c\}$. Without loss of generality, we may assume that $z_l$ is the unique vertex in $V(H)\backslash \{ c\}$ on the path. The strategy is to find a "smaller" subgraph $H'$ of type I to substitute $H$. Here by "smaller" we mean the number $N$ in the cyclic ordering $(ca_1,cx_1,\cdots,cx_N,cd_1,cy_1,\cdots,cy_M)$ is smaller. There are $5$ cases based on the different positions of $z_l$ and sometimes the cyclic ordering around $z_l$. (See Figure~\ref{fig5})

\begin{figure}[h]
            \centering
            \includegraphics[scale=0.6]{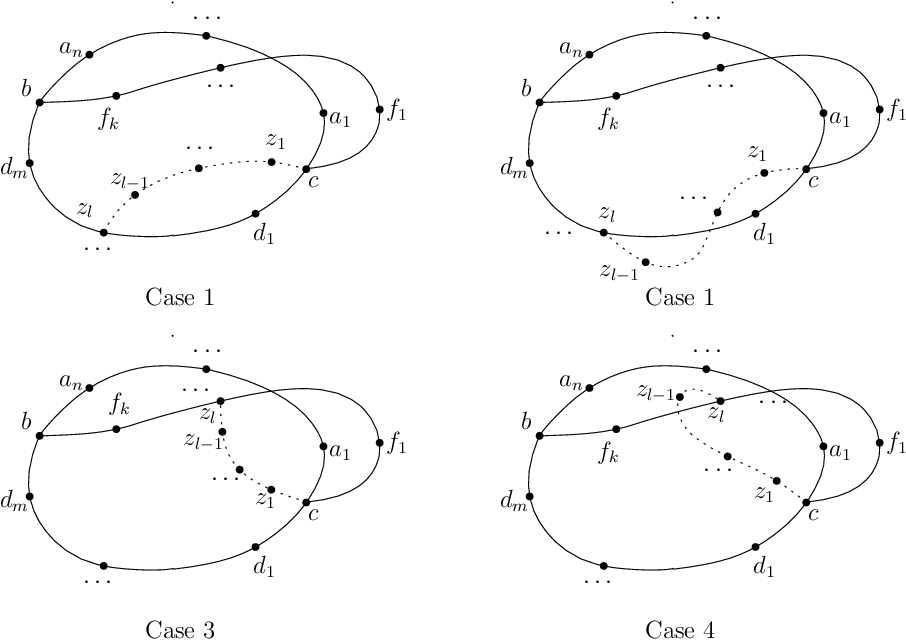}
            \caption{Some cases in the proof of Lemma~\ref{1toa}. The two different positions of $z_{l-1}z_l$ give two configurations in Case 1, but they result in the same way of constructing $H'$.}
            \label{fig5}
\end{figure}

\item Case 1: $z_l=d_j, j\in [m]$.

In this case, the subgraph $H'$ can be obtained by replacing the path $(c,d_1,\cdots,z_l)$ in $H$ with the path $(z_0,z_1,z_2,\cdots,z_l)$.  

\item Case 2: $z_l=a_j, j\in [n]$. 

This case is similar to Case 1 because the graph $H$ is "symmetric" about the path $(c,f_1,\cdots,f_k,b)$.

\item Case 3: $z_l=f_j, j\in [k]$, and the cyclic ordering of edges $\{f_jz_{l-1},f_jf_{j-1}, f_jf_{j+1}\}$ around $f_j$ is $(f_jz_{l-1},f_jf_{j-1},f_jf_{j+1})$.

In this case, the subgraph $H'$ can be obtained by replacing the path $(c,d_1,\cdots,d_m,b)$ in $H$ with the path $(z_0,z_1,z_2,\cdots,z_l)$.  

\item Case 4: $z_l=f_j, j\in [k]$, and the cyclic ordering of edges $\{f_jz_{l-1},f_jf_{j-1}, f_jf_{j+1}\}$ around $f_j$ is $(f_jz_{l-1},f_jf_{j+1},f_jf_{j-1})$.

This case is similar to Case 3 because the graph $H$ is "symmetric" about the path $(c,f_1,\cdots,f_k,b)$.

\item Case 5: $z_l=b$. 

This case can be viewed as a special case of Case 1 or Case 2.

In all these cases, the new subgraph $H'$ is of type I and $N$ decreases strictly. Hence by repeating this process, we can get a subgraph $H$ satisfying the second part of Definition~\ref{type1a}. 
\end{proof}

\begin{Lemma}\label{2tob}
If a non-planar ribbon graph $G$ does not contain a subgraph $H$ of type I, then $G$ is of type B.
\end{Lemma}

\begin{proof}
By Lemma~\ref{type12} the graph $G$ must contain a subgraph $H$ of type II. Let $G_1$ be as in the first part of Definition~\ref{type2b}. If $G_1$ does not contain any of the vertices $a_1,\cdots, a_n$, then $G$ is already of type B.
Otherwise, by the construction of $G_1$ there exists a path with vertex sequence $(z_0,z_1,z_2,\cdots,z_l)$ where $z_0=c,z_1\in\{x_1,\cdots,x_N\}$, $z_1,z_2,z_3,\cdots,z_{l-1}\neq c$, and $z_l\in\{a_1,\cdots,a_n\}$. Without loss of generality, we may assume that $z_l$ is the unique vertex in $\{a_1,\cdots,a_n\}$ on the path. Similar to the proof of Lemma~\ref{1toa}, the strategy is to find a "smaller" subgraph $H'$ of type II to substitute $H$. 

The key fact here is that the path cannot contain any of the vertices $f_1,\cdots,f_k$. Otherwise, we can get a subgraph $\widetilde{H}$ of type I, which leads to a contradiction. Indeed, let $z_j$ be the last vertex in the set $\{f_1,\cdots,f_k\}$ on the path and say $z_j=f_{j'}$. Then we consider the subpath $L=(z_j,z_{j+1},\cdots,z_l)$(see Figure~\ref{fig61}). If the cyclic ordering of edges $\{z_jz_{j+1}, f_{j'}f_{j'+1}, f_{j'}f_{j'-1}\}$ around $z_j(=f_{j'})$ is $(z_jz_{j+1}, f_{j'}f_{j'+1}, f_{j'}f_{j'-1})$, then we can obtain $\widetilde{H}$ of type I by replacing the path $(c,a_n,\cdots,z_l)$ in $H$ with the path $L$. If the cyclic ordering of edges $\{z_jz_{j+1}, f_{j'}f_{j'+1}, f_{j'}f_{j'-1}\}$ around $z_j(=f_{j'})$ is $(z_jz_{j+1}, f_{j'}f_{j'-1}, f_{j'}f_{j'+1})$, then we can obtain $\widetilde{H}$ of type I by replacing the path $(c,a_1,\cdots,z_l)$ in $H$ with the path $L$.  

\begin{figure}[h]
            \centering
            \includegraphics[scale=0.6]{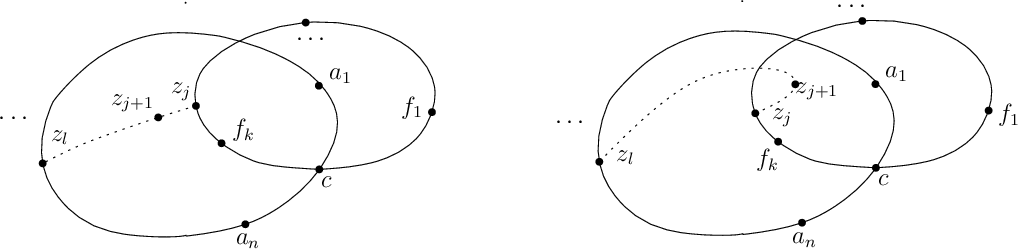}
            \caption{Figures used for the construction of $\widetilde{H}$ in Lemma~\ref{2tob}}
            \label{fig61}
\end{figure}

The remaining part of the proof is similar to the proof of Lemma~\ref{1toa}. Set $f_k=x_i,z_1=x_j$ for some $i,j\in [N]$(see Figure~\ref{fig62}). Note that $i\neq j$ because $f_k$ is not in the path. If $j>i$, we can get $H'$ by replacing the path $(c,a_n,\cdots,z_l)$ in $H$ by the path $(z_0,z_1,z_2,\cdots,z_l)$; if $j<i$, we can get $H'$ by replacing the path $(c,a_1,\cdots,z_l)$ in $H$ by the path $(z_0,z_1,z_2,\cdots,z_l)$. In both cases, $H'$ is of type II and the number $N$ in the cyclic ordering $(ca_1,cx_1,\cdots,cx_N,ca_n,cy_1,\cdots,cy_M)$ decreases strictly. So by repeating this process, we can get a subgraph $H$ making $G$ of type B.  

\begin{figure}[h]
            \centering
            \includegraphics[scale=0.6]{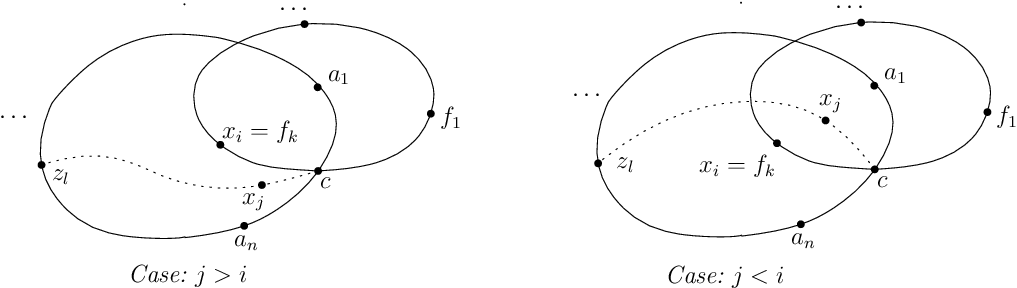}
            \caption{Figures used for the construction of $H'$ in Lemma~\ref{2tob}}
            \label{fig62}
\end{figure}

\end{proof}

\begin{Cor}\label{AB}
Any non-planar ribbon graph is of type A or of type B (or of both types). 
\end{Cor}
\begin{proof}
This is a direct consequence of Lemma~\ref{1toa} and Lemma~\ref{2tob}.
\end{proof}

\section{Proof of the main result}\label{Proof}
In this section we present the proof of Theorem~\ref{main}. It consists of two parts: one is for type A and the other is for type B. 

We first state a basic lemma. Let $G$ be a graph. Assume $G=G_1 \vee_c G_2$, where $G_1$ and $G_2$ are two subgraphs. Because $\text{Div}^0(G)$ is freely generated by $\{(v)-(c): v\in V(G)\backslash\{c\}\}$, we have group isomorphisms 
$$
\text{Div}^0(G)\simeq 
\bigoplus_{v\in V(G)\backslash\{c\}}\mathbb{Z}((v)-(c))\simeq 
\text{Div}^0(G_1)\oplus\text{Div}^0(G_2),
$$
and the composition  $\phi:\text{Div}^0(G)\longrightarrow\text{Div}^0(G_1)\oplus\text{Div}^0(G_2)$ is $D\mapsto (D_1,D_2)$ such that $D$ can be uniquely written as $D=D_1+D_2$, where $D_i\in\text{Div}^0(G_i)$, $i=1,2$. Furthermore, we have the following isomorphism.

\begin{Lemma}\label{l4}
The map $\phi$ descents to an isomorphism  $\overline{\phi}:\text{Pic}^0(G)\longrightarrow\text{Pic}^0(G_1)\oplus\text{Pic}^0(G_2)$. 
\end{Lemma}
\begin{proof}
The proof is left to the reader. 
\end{proof}

\noindent\textbf{Notation:} Given a decomposition $G=G_1 \vee_c G_2$, we denote by $\sim_i$ the linear equivalence relation with respect to $G_i$, $i=1,2$. Similarly, we denote by $\beta^i_{(v_i,e_i)}$ the Bernardi bijection with respect to $G_i$, $i=1,2$. 

\begin{Rem}\label{remark}
Let $G=G_1 \vee_c G_2$ and $D=D_1+D_2$ be a divisor of $G$, where $D_i\in\text{Div}^0(G_i)$, $i=1,2$. By Lemma~\ref{l4}, in order to show $D\nsim 0$, it suffices to prove that $D_1\nsim_1 0$ or $D_2\nsim_2 0$.

\end{Rem}

We are now ready to prove Theorem~\ref{main} for ribbon graphs of type A.

\begin{Prop}\label{A}
If a non-planar ribbon graph $G$ is of type A, then the Bernardi and rotor-routing torsors $b_v$ and $r_v$ do not agree for some vertex $v$ of $G$.
\end{Prop}

\begin{proof}
Let $G$ and $(H,G_1,G_2)$ be as in Definition~\ref{type1a}. Recall that the cyclic ordering of the edges around $c$ in $G$ is denoted by $(ca_1,cx_1,\cdots,cx_N,cd_1,cy_1,\cdots,cy_M)$. Because $G$ is of type $A$, $G_2$ contains $H$ and $G_1$ contains all the edge $cx_i$'s. We want to prove that $b_{d_1}\neq r_{d_1}$.

Let $T$ be a spanning tree of $G$ that contains every edge in $H$ except $cd_1$ and $bf_k$ (see Figure~\ref{fig7-1}). Let $d_1$ be the sink and put the chip at $c$. Then by applying the rotor-routing process, we get another spanning tree $T'=((c)-(d_1))_{d_1}(T)$. By Lemma~\ref{l1}, it suffices to show 
$$\beta_{(d_1,d_1c)}(T')-\beta_{(d_1,d_1c)}(T)\nsim (c)-(d_1).$$
Let $T=T_1 \vee_c T_2$, where $T_i$ is obtained by restricting $T$ to $G_i$ ($i=1,2$).  Similarly, let  $T'=T'_1 \vee_c T'_2$, where $T'_i$ is obtained by restricting $T'$ to $G_i$ ($i=1,2$). 

\begin{figure}[h]
            \centering
            \includegraphics[scale=0.6]{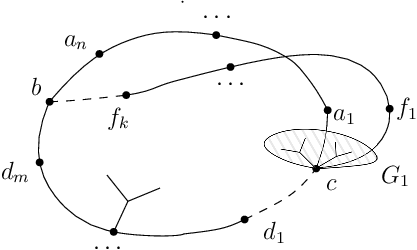}
            \caption{The spanning tree $T$. An edge is dashed if and only if it is not in the tree. The part of the tree not in $H$ is drawn in an abstract way.}
            \label{fig7-1}
\end{figure}

We claim that $T_2'=T_2\cup\{cd_1\}\backslash\{ca_1\}$ (see Figure~\ref{fig7-2}). It is because in the rotor-routing process after the chip goes into $G_1$, the process does not affect $G_2$ until the chip quits $G_1$ along the edge $cd_1$ and thereby reaches the sink $d_1$. 

\begin{figure}[h]
            \centering
            \includegraphics[scale=0.6]{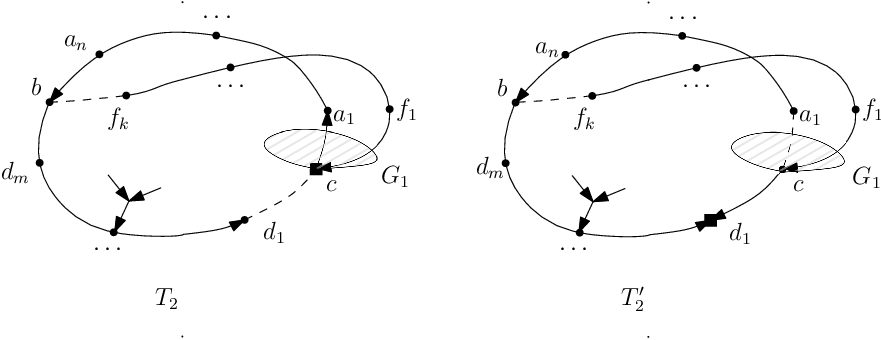}
            \caption{The rotor-routing process for $T'=((c)-(d_1))_{d_1}(T)$. The oriented edges indicate the rotor configuration and the square indicates the chip. The rotor configuration inside $G_1$ is omitted.}
            \label{fig7-2}
\end{figure}

Now consider $\beta_{(d_1,d_1c)}(T')$ and $\beta_{(d_1,d_1c)}(T)$. We have 
$$\beta_{(d_1,d_1c)}(T')-\beta_{(d_1,d_1c)}(T)=\beta^2_{(d_1,d_1c)}(T_2')-\beta^2_{(d_1,d_1c)}(T_2)+D,$$
where the divisor $D$ is the part of $\beta_{(d_1,d_1c)}(T')-\beta_{(d_1,d_1c)}(T)$ contributed by $G_1$. 

By Lemma~\ref{l4} (and Remark~\ref{remark}), if $\beta^2_{(d_1,d_1c)}(T_2')-\beta^2_{(d_1,d_1c)}(T_2)
\nsim_2 (c)-(d_1)$, then $\beta_{(d_1,d_1c)}(T')-\beta_{(d_1,d_1c)}(T)\nsim (c)-(d_1)$ as desired.

By comparing the Bernardi tours of $T_2$ and $T_2'$ (see Figure~\ref{fig7-3}), we will write  
\begin{equation}\label{eq0}
    \beta^2_{(d_1,d_1c)}(T_2')-\beta^2_{(d_1,d_1c)}(T_2)=(c)-(d_1)+\sum_{\overrightarrow{e}\in B}\partial \overrightarrow{e},
\end{equation}
where $B$ is a partial orientation of $G_2$ which we now define. The following 6-step process describes the tour of $T_2$ and we regroup them into four parts. 
\begin{enumerate}
\item The tour cuts the edge $cd_1$ at $d_1$. (Part I)
\item This step starts right after the end of the previous step and ends right before the tour visits the edge $a_1c$ for the first time. (Part II)
\item The tour goes along the edge $a_1c$ and then cuts $cd_1$ at $c$. (Part I)
\item This step starts right after the end the previous step and ends right before the tour visits the edge $a_1c$ for the second time.  (Part III)
\item The tour goes along $ca_1$ without cutting any edge. (Part I)
\item This step finishes the remaining part of the tour. (Part IV)
\end{enumerate}

Similarly, the following 6-step process describes the tour of $T'_2$ and we regroup them into four parts. 
\begin{enumerate}
\item The tour goes along the edge $d_1c$ without cutting any edge. (Part I')
\item This step starts right after the end the previous step and ends right before the tour cuts the edge $ca_1$ for the first time. (Part III)
\item The tour cuts the edge $ca_1$ at $c$ and then goes along $cd_1$ without cutting any edge. (Part I')
\item This step starts right after the end of the previous step and ends right before the tour cuts the edge $ca_1$ for the second time. (Part II)
\item The tour cuts the edge $ca_1$ at $a_1$. (Part I')
\item This step finishes the remaining part of the tour. (Part IV)
\end{enumerate}

The key observation is that the two tours share Parts II, III, IV, but the order of visits of Parts II and III is exchanged. Now we can calculate the left hand side of \eqref{eq0}. One finds that Parts II, III, IV contribute the divisor $\sum_{\overrightarrow{e}\in B}\partial \overrightarrow{e}$, where $B$ is the partial orientation of $G_2$ consisting of the arcs going from Part II to Part III, and the difference between Part I and Part I' contributes the divisor $(c)-(d_1)$.  So \eqref{eq0} holds.

\begin{figure}[h]
            \centering
            \includegraphics[scale=0.6]{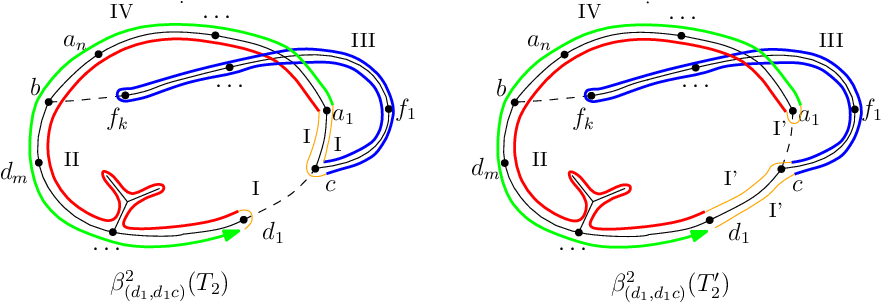}
            \caption{The Bernardi tours of $T_2$ and $T'_2$. Parts II, III, and IV are indicated in red, blue, and green respectively.}
            \label{fig7-3}
\end{figure}

It remains to show $\sum_{\overrightarrow{e}\in B}\partial \overrightarrow{e}\nsim_2 0$. Assume by contradiction that $\sum_{\overrightarrow{e}\in B}\partial \overrightarrow{e}\sim_2 0$. Then by Lemma~\ref{l3}, $B$ is a disjoint union of directed cuts in $G_2$. Note that $B$ contains at least one arc $\overrightarrow{bf_k}$, so $B$ contains a directed cut and hence $G_2\backslash \{\text{edges in }  B\}$ is disconnected. This contradicts the fact that $T_2$ is a spanning tree of $G_2$.  

\end{proof}

It remains to prove Theorem~\ref{main} for ribbon graphs of type B, which is Proposition~\ref{B}. We need the following lemma. 

\begin{Lemma}\label{l5}
Let $G$ be a graph of type B and $(H, G_1, G_2)$ be as in Definition~\ref{type2b}. Recall that the cyclic ordering of the edges around the vertex $c$ in $G$ is denoted by $(ca_1,cx_1,\cdots,cx_N,ca_n,cy_1,\cdots,cy_M)$ and all the edges $cx_i$'s are in $G_1$. We further assume that $G$ has no loops or multiple edges, then $2\sum_{i=1}^{N}\partial(\overrightarrow{cx_i})\nsim_1 0$.   
\end{Lemma}
\begin{proof}
All the following arguments and calculations are made with respect to $G_1$. 

Note that $2\sum_{i=1}^{N}\partial(\overrightarrow{cx_i})\sim_1 \sum_{\overrightarrow{e}\in B}\partial \overrightarrow{e}$, where $B$ is the partial orientation $\{\overrightarrow{cx_i}:i\in [N]\}\cup\{\overrightarrow{y_jc}:cy_j\in E(G_1), j\in [M]\}$. 
Assume by contradiction that $\sum_{\overrightarrow{e}\in B}\partial \overrightarrow{e}\sim_1 0$. Because $G_1$ has no loops or multiple edges, $B$ contains no directed cycle. By Lemma~\ref{l3}, $B$ is a disjoint union of directed cuts of $G_1$. 

In particular, $\overrightarrow{cf_k}$ belongs to a directed cut $\overrightarrow{C}$ of $G_1$, which is a subset of $B$. Because the intersection of a cut and a cycle contains at least two edges, so there is another edge in the cycle $\{cf_k,f_kf_{k-1},\cdots,f_2f_1,f_1c\}$ distinct from $cf_k$ in the cut $\overrightarrow{C} \subseteq B$. The only candidate is $f_1c$ and hence $\overrightarrow{f_1c}\in\overrightarrow{C}$. However, $\overrightarrow{f_1c}$ and $\overrightarrow{cf_k}$ cannot be in one directed cut. We reach a contradiction.

\end{proof}

Now we are ready to prove the following proposition. The proof is technical, so we give an example (Example~\ref{ex2}) after the proof. 

\begin{Prop}\label{B}
If a non-planar ribbon graph $G$ without loops or multiple edges is of type B, then the Bernardi and rotor-routing torsors $b_v$ and $r_v$ do not agree for some vertex $v$ of $G$.
\end{Prop}

\begin{proof}
Let $G$ and $(H,G_1,G_2)$ be as in Definition~\ref{type2b}. Recall that the cyclic ordering of the edges around $c$ in $G$ is denoted by $(ca_1,cx_1,\cdots,cx_N,ca_n,cy_1,\cdots,cy_M)$. Because $G$ is of type B, $G_2$ contains the cycle $\{ca_1, a_1a_2, \cdots, a_nc\}$ and $G_1$ contains all the edge $cx_i$'s. We want to prove that $b_{a_n}\neq r_{a_n}$ or $b_{c}\neq r_{c}$.

Let $T$ be a spanning tree of $G$ that contains $\{ca_1, a_1a_2, \cdots, a_{n-1}a_n\}$ (see Figure~\ref{fig8-1}). Let $a_n$ be the sink and put the chip at $c$. Then by applying the rotor-routing process, we get another spanning tree $T'=((c)-(a_n))_{a_n}(T)$. 

Let $T=T_1 \vee_c T_2$, where $T_i$ is obtained by restricting $T$ to $G_i$ ($i=1,2$).  Similarly, let  $T'=T'_1 \vee_c T'_2$, where $T'_i$ is obtained by restricting $T'$ to $G_i$ ($i=1,2$).

For the same reason as in the proof of Proposition~\ref{A}, we have $T_2'=T_2\cup\{ca_n\}\backslash\{ca_1\}$.

\begin{figure}[h]
            \centering
            \includegraphics[scale=0.6]{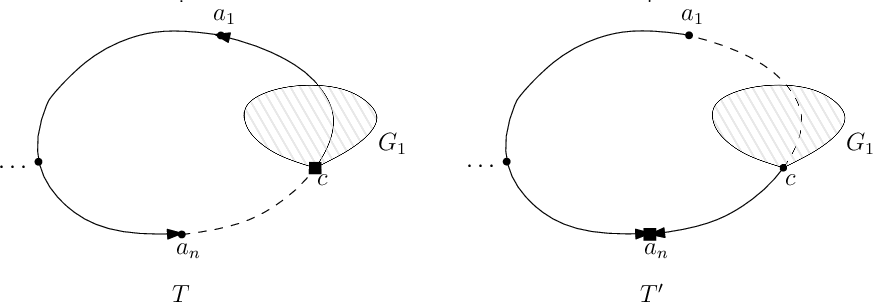}
            \caption{The rotor-routing process for $T'=((c)-(a_n))_{a_n}(T)$. An edge is dashed if and only if it is not in the tree. The oriented edges indicate the rotor configuration and the square indicates the chip. We only show the rotor configuration for the edges in the cycle $\{ca_1, a_1a_2, \cdots, a_nc\}$.}
            \label{fig8-1}
\end{figure}

\begin{figure}[h]
            \centering
            \includegraphics[scale=0.6]{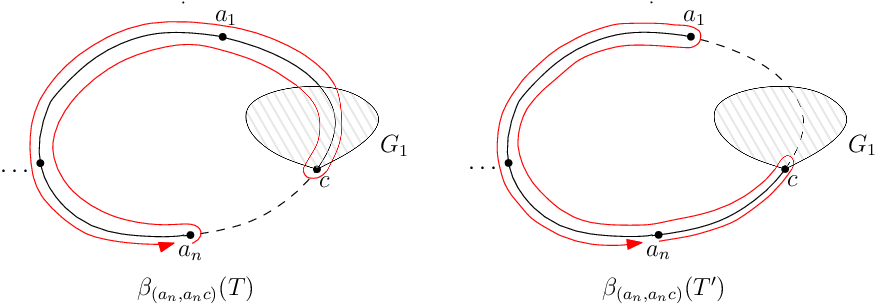}
            \caption{The Bernardi tours of $T$ and $T'$. The key information is that the tour of $T$ ($T'$) enters $G_1$ via the edge $cx_1$ ($cy_1$).}
            \label{fig8-2}
\end{figure}

By applying the Bernardi process (see Figure~\ref{fig8-2}), we have $$\beta_{(a_n,a_nc)}(T')-\beta_{(a_n,a_nc)}(T)=\beta^2_{(a_n,a_nc)}(T'_2)-\beta^2_{(a_n,a_nc)}(T_2)+\beta^1_{(c,cy_1)}(T'_1)-\beta^1_{(c,cx_1)}(T_1),$$
because the Bernardi process for $G_1$ and the one for $G_2$ are independent.


We focus on $G_1$. We rewrite the above identity as
$$\beta_{(a_n,a_nc)}(T')-\beta_{(a_n,a_nc)}(T)=(c)-(a_n)+D+\beta^1_{(c,cy_1)}(T'_1)-\beta^1_{(c,cx_1)}(T_1),$$
where $D=\beta^2_{(a_n,a_nc)}(T'_2)-\beta^2_{(a_n,a_nc)}(T_2)-(c)+(a_n)\in\text{Div}^0(G_2)$.

By Lemma~\ref{l1} and Lemma~\ref{l4}, if $\beta^1_{(c,cy_1)}(T'_1)-\beta^1_{(c,cx_1)}(T_1)\nsim_1 0$ in $G_1$, then $b_{a_n}\neq r_{a_n}$.

Now we begin to compute $\beta^1_{(c,cy_1)}(T'_1)-\beta^1_{(c,cx_1)}(T_1)$. By Lemma~\ref{l2}, 
$$\beta^1_{(c,cy_1)}(T'_1)-\beta^1_{(c,cx_1)}(T'_1)\sim_1\sum_{i=1}^{N}\partial(\overrightarrow{cx_i}),$$
so we have
\begin{equation}\label{eq1}
    \beta^1_{(c,cy_1)}(T'_1)-\beta^1_{(c,cx_1)}(T_1)\sim_1 \beta^1_{(c,cx_1)}(T'_1)-\beta^1_{(c,cx_1)}(T_1)+\sum_{i=1}^{N}\partial(\overrightarrow{cx_i}).
\end{equation}

In order to calculate $\beta^1_{(c,cx_1)}(T'_1)-\beta^1_{(c,cx_1)}(T_1)$ in the above formula, we need to elaborate how $T_1$ becomes $T'_1$ in the rotor-routing process for $T'=((c)-(a_n))_{a_n}(T)$. Initially, the chip is at the vertex $c$ and the rotor configuration is $T_1\cup T_2$. Then the chip goes to the vertex $x_1$ and the rotor configuration $T_1\cup T_2$ becomes $T_1\cup\{cx_1\}\cup T_2\backslash\{ca_1\}$. The key observation here is that, when the chip goes back to $c$, the tree $T_1$ in the rotor configuration becomes the tree $((x_1)-(c))_c(T_1)$, where the rotor-routing action is applied with respect to $G_1$, and hence the rotor configuration becomes $((x_1)-(c))_c(T_1)\cup\{cx_1\}\cup T_2\backslash\{ca_1\}$. For the next step, the chip goes to $x_2$ and the rotor configuration becomes $((x_1)-(c))_c(T_1)\cup\{cx_2\}\cup T_2\backslash\{ca_1\}$. In general, we denote $T_1^{(1)}:=T_1$ and $T_1^{(i+1)}:=((x_i)-(c))_c(T_1^{(i)})$ for $i\in [N]$, and hence get Table~\ref{table}. Note that $T'_1=T_1^{(N+1)}$.

\begin{table}[h]
\begin{center}
\begin{tabular}{|c|c|c|}
\hline
Position of the Chip  & The Rotor Configuration & Remark \\ \hline 
$c$  & $T_1\cup T_2$ &  \\ \hline
$x_1$  & $T_1^{(1)}\cup\{cx_1\}\cup T_2\backslash\{ca_1\}$ & $T_1^{(1)}=T_1$ \\ \hline
$c$  & $T_1^{(2)}\cup\{cx_1\}\cup T_2\backslash\{ca_1\}$ &  $T_1^{(2)}=((x_1)-(c))_c(T_1^{(1)})$\\ \hline
$x_2$  & $T_1^{(2)}\cup\{cx_2\}\cup T_2\backslash\{ca_1\}$ &  \\ \hline
$\cdots$  & $\cdots$ &  $T_1^{(i+1)}=((x_i)-(c))_c(T_1^{(i)})$\\ \hline
$x_N$  &$T_1^{(N)}\cup\{cx_N\}\cup T_2\backslash\{ca_1\}$ &  \\ \hline
$c$  & $T_1^{(N+1)}\cup\{cx_N\}\cup T_2\backslash\{ca_1\}$ & $T_1^{(N+1)}=((x_N)-(c))_c(T_1^{(N)})$ \\ \hline
$a_n$  & $T'_1\cup T'_2$ & $T'_1=T_1^{(N+1)}$ \\ \hline
\end{tabular}
\end{center}
\caption{The rotor-routing process for $T'=((c)-(a_n))_{a_n}(T)$. The rotor-routing actions in the last column are applied with respect to $G_1$.}
\label{table}
\end{table}
For the final part of the proof, we discuss two cases.

\item Case 1: In $G_1$, $\beta^1_{(c,cx_1)}(T_1^{(i+1)})-\beta^1_{(c,cx_1)}(T_1^{(i)})\sim_1 (x_i)-(c)$ for all $i\in [N]$. 

In this case, 
$$\beta^1_{(c,cx_1)}(T'_1)-\beta^1_{(c,cx_1)}(T_1)=\sum_{i=1}^N(\beta^1_{(c,cx_1)}(T_1^{(i+1)})-\beta^1_{(c,cx_1)}(T_1^{(i)}))\sim_1 \sum_{i=1}^{N}\partial(\overrightarrow{cx_i}).$$
Together with \eqref{eq1}, we get 
\begin{equation}\label{eq2}
\beta^1_{(c,cy_1)}(T'_1)-\beta^1_{(c,cx_1)}(T_1)\sim_1 2\sum_{i=1}^{N}\partial(\overrightarrow{cx_i}).
\end{equation}
Then we apply Lemma~\ref{l5} and conclude that  $b_{a_n}\neq r_{a_n}$.

\item Case 2:  In $G_1$, $\beta^1_{(c,cx_1)}(T_1^{(i+1)})-\beta^1_{(c,cx_1)}(T_1^{(i)})\nsim_1 (x_i)-(c)$ for some $i\in [N]$. 

In this case, we will show that $b_{c}\neq r_{c}$. Consider the spanning tree $T_i^{(i)}\vee_c T_2$ of $G$, let $c$ be the sink, and put the chip at $x_i$. (Here $T_2$ can be replaced by any spanning tree of 
$G_2$.)

By applying the rotor-routing process, we have
$$((x_i)-(c))_c(T_1^{(i)}\vee_c T_2)=T_1^{(i+1)}\vee_c T_2.$$

Running the Bernardi process, we get
$$\beta_{(c,cx_1)}(T_1^{(i+1)}\vee_c T_2)=\beta^1_{(c,cx_1)}(T_1^{(i+1)})+\beta^2_{(c,ca_n)}(T_2),$$
$$\beta_{(c,cx_1)}(T_1^{(i)}\vee_c T_2)=\beta^1_{(c,cx_1)}(T_1^{(i)})+\beta^2_{(c,ca_n)}(T_2),$$
and hence
$$\beta_{(c,cx_1)}(T_1^{(i+1)}\vee_c T_2)-\beta_{(c,cx_1)}(T_1^{(i)}\vee_c T_2)=\beta^1_{(c,cx_1)}(T_1^{(i+1)})-\beta^1_{(c,cx_1)}(T_1^{(i)})\nsim (x_i)-(c),$$
where the last step is due to the assumption of Case 2 and Lemma~\ref{l4}.

By Lemma~\ref{l1}, this implies that $b_{c}\neq r_{c}$.
\end{proof}

\begin{Ex}\label{ex2}
Here we give an example to demonstrate how the proof of Proposition~\ref{B} works. In Figure~\ref{fig9-1} and Figure~\ref{fig9-2}, the graph $G$ consists of 6 vertices and 6 edges and it is of type B. Let the spanning tree $T$ be $\{ca_1,a_1a_2,cf_1,cf_2\}$. Then $T_1=\{cf_1,cf_2\},T_2=\{ca_1,a_1a_2\}$. Focusing on $G_2$, one can see that $T'_2$ is $\{ca_2,a_1,a_2\}$, whatever the tree $T_1$ is. Focusing on $G_1$, one can get $T'_1=T_1^{(2)}=((f_2)-(c))_c(T_1^{(1)})=((f_2)-(c))_c(T_1)$, which is exactly Step 1 to Step 4 by ignoring $\overrightarrow{cf_2}$. Because $G_1$ is planar, the two torsors agree at $c$ and hence the example belongs to Case 1 in the final part of the proof. A direct calculation shows that $\beta_{(a_1,a_2c)}(T')-\beta_{(a_1,a_2c)}(T)=((c)-(a_2))+((c)-(f_2))$. The first summand $(c)-(a_2)$ is contributed by $G_2$ and the second summand $(c)-(f_2)$ is contributed by $G_1$. By \eqref{eq2}, $(c)-(f_2)$ should be linearly equivalent to $2\partial(\overrightarrow{cf_2})$, which is true because one check that $3[(c)-(f_2)]=0$ in $\text{Pic}^0(G_1)$. Because $(c)-(f_2)\nsim 0$, the two torsors disagree at the vertex $a_2$.
\end{Ex}

\begin{figure}[h]
            \centering
            \includegraphics[scale=0.6]{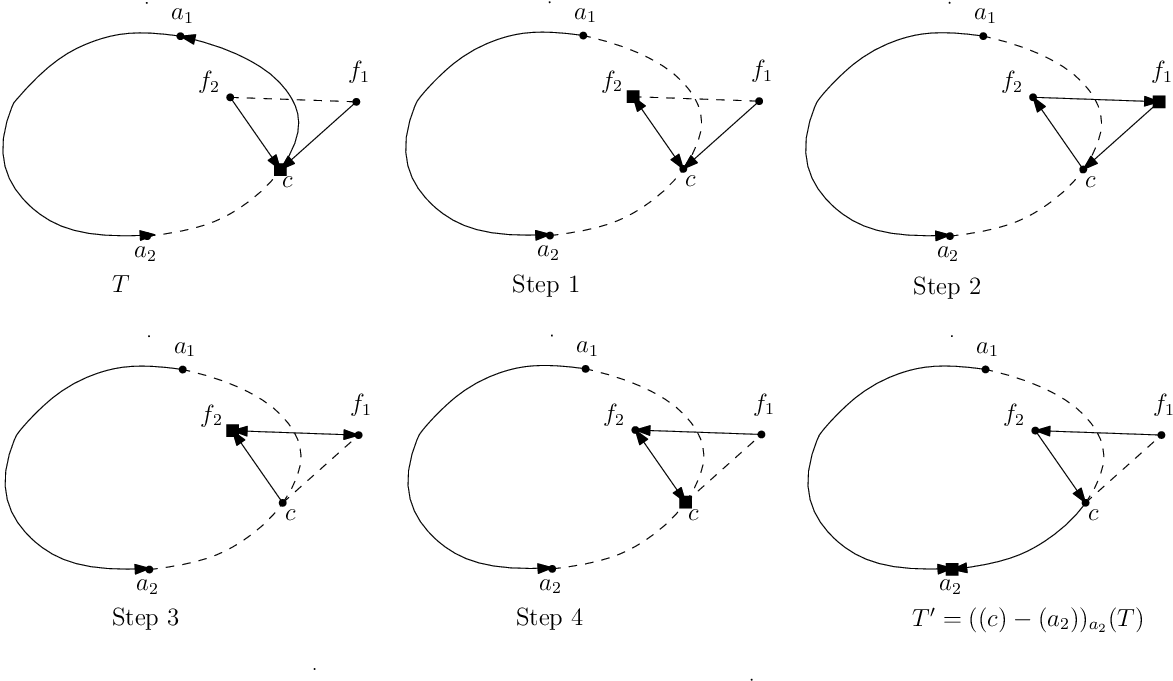}
            \caption{The rotor-routing process for $T'=((c)-(a_2))_{a_2}(T)$ in Example~\ref{ex2}. The oriented edges indicate the rotor configuration and the square indicates the chip.}
            \label{fig9-1}
\end{figure}

\begin{figure}[h]
            \centering
            \includegraphics[scale=0.6]{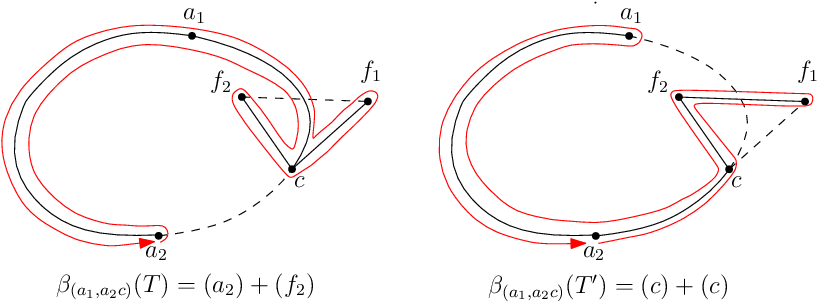}
            \caption{The Bernardi tours of $T$ and $T'$ in Example~\ref{ex2}.}
            \label{fig9-2}
\end{figure}

The main Theorem~\ref{main} is a direct consequence of Corollary~\ref{AB}, Proposition~\ref{A}, and Proposition~\ref{B}.

\begin{Rem}
In Theorem~\ref{main}, one cannot remove the assumption that $G$ has no loops or multiple edges. Indeed, loops are invisible to both the rotor-routing torsor and the Bernardi torsor. However, one can make a planar ribbon graph non-planar by adding loops to it, so we must assume $G$ is loopless in the theorem. Let $G$ be the graph of type II with $k=n=1$ (See Figure~\ref{fig2}). One can check that the two torsors agree at any vertex of $G$, although $G$ is non-planar. So we must also assume $G$ has no multiple edges in the theorem. In our proof, the first place where we need this assumption is Lemma~\ref{l5}. 

\end{Rem}

\section*{Acknowledgement}

Many thanks to Olivier Bernardi for orienting the author towards the study of this conjecture and countless helpful discussions, as well as detailed advice on the writing of this paper. As a graduate student, the author wants to thank the Department of Mathematics at Brandeis University for admitting the author to the Ph.D. program and creating a flexible study environment. The author also thanks the anonymous referee for the helpful feedback.

\end{document}